\documentclass{amsart}

\usepackage{amscd, amssymb, amsmath, amsthm}
\usepackage{enumerate, latexsym, mathrsfs}
\usepackage{xypic}
\usepackage[
		bookmarks=true, bookmarksopen=true,%
    bookmarksdepth=3,bookmarksopenlevel=2,%
    colorlinks=true,%
    linkcolor=blue,%
    citecolor=blue]{hyperref}
\newtheorem{theorem}{Theorem}[section]
\newtheorem{lemma}[theorem]{Lemma}

\theoremstyle{definition}
\newtheorem{definition}[theorem]{Definition}

\newtheorem{remark}[theorem]{Remark}

\numberwithin{equation}{section}

\newcommand{\Natural}{{\mathbb N}}
\newcommand{\Real}{{\mathbb R}}

\newcommand{\Complex}{{\mathbb C}}
\newcommand{\Integral}{{\mathbb Z}}

\newcommand{\ud}{{\mathrm{d}}}

\newcommand{\tlen}{{\ell}}

\title[{Virtual homological eigenvalues and the WP translation length}]{
Virtual homological eigenvalues and\\ the Weil--Petersson translation length}
     
\author[Yi Liu]{%
        Yi Liu} 
\address{%
        Beijing International Center for Mathematical Research, Peking University\\
				Beijing 100871, China P.R.} 
\email{%
    liuyi@bicmr.pku.edu.cn}

\thanks{Partially supported by NSFC Grant 11925101, 
and National Key R\&D Program of China 2020YFA0712800}
\subjclass[2020]{Primary 57K20; Secondary 57K31}
\keywords{homological eigenvalue, finite cover, Weil--Petersson metric, translation length}

\date{%
 \today} 

\begin{document}

\begin{abstract}
	For any pseudo-Anosov automorphism on an orientable closed surface,
	an inquality is established bounding 
	certain growth of virtual homological eigenvalues with the Weil--Petersson translation length.
	The new inquality fits nicely with other known inequalities 
	due to Kojima and McShane, and due to L\^{e}.	
	
	The new quantity to be considered is 
	the square sum of the logarithmic radii of 
	the homological eigenvalues (with multiplicity)
	outside the complex unit circle,
	called the homological \emph{Jensen square sum}.
	The main theorem is as follows.
	For any cofinal sequence of regular finite covers of a given surface,
	together with lifts of a given pseudo-Anosov,
	the homological Jensen square sum of the lifts
	grows at most linearly fast compared to the covering degree,
	and the square root of the growth rate is at most 
	$1/\sqrt{4\pi}$ times
	the Weil--Petersson translation length of the given pseudo-Anosov.	
\end{abstract}

\maketitle

\section{Introduction}
Let $S$ be a connected closed orientable surface of genus at least $2$.
Let $f\colon S\to S$ be a pseudo-Anosov automorphism.
In \cite{Kojima--McShane},
S.~Kojima and G.~McShane establish an elegant inequality,
which we can rearrange as
\begin{equation}\label{KM_inequality}
\frac{\mathrm{Vol}_{\mathtt{h}}(M_f)}{6\pi}\leq
\frac{\ell_{\mathtt{WP}}(f)}{\sqrt{4\pi}}\times \sqrt{\mathrm{genus}(S)-1}.
\end{equation}
In this inequality,
$\mathrm{Vol}_{\mathtt{h}}(M_f)$ denotes the hyperbolic volume of the mapping torus $M_f$
of $f$, with respect to its isometrically unique hyperbolic geometry,
and $\ell_{\mathtt{WP}}(f)$ denotes the Weil--Petersson translation length of $f$,
as a mapping class acting on the Teichm\"uller space of $S$.

The left-hand side of (\ref{KM_inequality}) is equal to 
the logarithmic $L^2$ torsion of the orientable closed $3$--manifold $M_f$; 
see \cite[Chapter 3]{Lueck_book} (with a different sign convention).
This quantity is a known upper bound for the superior linear growth rate
of the virtual homological logarithmic Mahler measures,
with respect to
any cofinal sequence of regular finite covers $(S'_n)_{n\in\Natural}$ of $S$
and lifts $(f'_n)_{n\in\Natural}$ of $f$.
To be precise, 
if $\kappa_n\colon S'_n\to S$ are cofinal regular finite covering maps
(that is, the corresponding finite-index normal subgroups of $\pi_1(S)$ 
have trivial common intersection),
and 
if $f'_n\colon S'_n\to S'_n$ are coverings of $f$
(that is, $\kappa_n\circ f'_n=f\circ \kappa_n$),
then the following inequality holds,
\begin{equation}\label{Le_pA}
\varlimsup_{n\to\infty} \frac{\mathbf{m}\left(P_n\right)}{[S'_n:S]} \leq \frac{\mathrm{Vol}_{\mathtt{h}}(M_f)}{6\pi}.
\end{equation}
In this inequality, $[S'_n:S]$ denotes the covering degree,
and $\mathbf{m}(P_n)$ denotes the logarithmic Mahler measure
of the characteristic polynomial $P_n$ 
of the induced linear automorphism $f'_{n*}\colon H_1(S'_n;\Complex)\to H_1(S'_n;\Complex)$.
In fact, (\ref{Le_pA}) follows easily from an inequality due to T.~T.~Q.~L\^{e} \cite{Le_hyperbolic_volume},
combined with \cite[Theorem 1]{Le_mahler_measure}.

We recall that for any nonzero complex polynomial $Q=Q(z)$,
the \emph{logarithmic Mahler measure} of $Q$ refers to the non-negative quantity
$$\mathbf{m}(Q)=\frac{1}{2\pi}\int_0^{2\pi} \log \left|Q\left(e^{\sqrt{-1}\cdot\theta}\right)\right|\,\ud\theta.$$
When $Q$ is monic of degree $d$ with complex roots $\lambda_1,\cdots,\lambda_d$,
the Jensen formula implies that $\mathbf{m}(Q)$ measures
exactly the total logarithmic radii of those roots outside the complex unit circle,
namely,
$$\mathbf{m}(Q)=\sum_{j=1}^d \log \max(1,\left|\lambda_j\right|).$$
See \cite[Chapter 1]{Everest--Ward_book}.

For any monic polynomial $Q$ of degree $d$ with roots $\lambda_1,\cdots,\lambda_d$, 
we introduce another non-negative quantity
\begin{equation}\label{w_def}
	\mathbf{w}(Q)=\sum_{j=1}^d\,\log^2\max\left(1,\left|\lambda_j\right|\right).
\end{equation}
In general, $\mathbf{w}(Q)$ is very different from the average of $\log^2|Q|$ over the unit circle.
The latter has been called the second higher Mahler measure; compare \cite{KLO_higher_mahler}.
Let us call $\mathbf{w}(Q)$ the \emph{Jensen square sum} of $Q$.

The polynomials $P_n$ in (\ref{Le_pA}) are monic of degree 
$2\times\mathrm{genus}(S'_n)$.
Moreover, they are all reciprocal.
We observe an elementary Cauchy--Schwartz inequality,
$$\frac{\mathbf{m}(P_n)}{[S'_n:S]}\leq 
\sqrt{\frac{\mathbf{w}(P_n)}{[S'_n:S]}}
\times\sqrt{\frac{\mathrm{genus}(S'_n)}{[S'_n:S]}},$$
where the limit of the last factor is clear,
$$\lim_{n\to\infty}\sqrt{\frac{\mathrm{genus}(S'_n)}{[S'_n:S]}}=
\sqrt{\mathrm{genus}(S)-1}.$$

In light of (\ref{KM_inequality}) and (\ref{Le_pA}),
there seems to be some unrevealed inequality
between the growth of $\mathbf{w}(P_n)$ and $\tlen_{\mathtt{WP}}(f)$.
In this paper, we establish that inequality,
exactly in the predicted form.

This is the following main theorem.

\begin{theorem}\label{main_vhev_WP}
	Let $f\colon S\to S$ be a pseudo-Anosov automorphism
	on a connected closed orientable surface of genus at least $2$.	
	Then, for any cofinal sequence of regular finite covers $(S'_n)_{n\in\Natural}$ of $S$
	with lifts $(f'_n)_{n\in\Natural}$ of $f$, 
	the following inequality holds,
	$$\varlimsup_{n\to\infty}\,\sqrt{\frac{\mathbf{w}(P_n)}{[S'_n:S]}}
	\leq
	\frac{\tlen_{\mathtt{WP}}(f)}{\sqrt{4\pi}},$$
	where $P_n$ denotes the characteristic polynomial	of $f'_{n*}$ on $H_1(S'_n;\Complex)$.
\end{theorem}

With notations of Theorem \ref{main_vhev_WP}, 
denote by $\mathbf{h}(P_n)$ the maximum of $\log \max(1,|\lambda|)$
where $\lambda$ ranges over all the complex roots of $P_n$.
Denote by $\tlen_{\mathtt{T}}(f)$ the Teichm{\"u}ller translation length of $f$,
which is equal to the entropy of $f$.
In \cite{McMullen_entropy}, C.~T.~McMullen shows that the well-known inequality
$$\varlimsup_{n\to\infty}\,\mathbf{h}(P_n)\leq \tlen_{\mathtt{T}}(f),$$
must be strict if the invariant foliations of $f$ have a prong singularity of odd order,
and moreover, in that case,
there is a uniform gap that depends only on $f$.
It seems reasonable to guess
that the inequality in Theorem \ref{main_vhev_WP} is strict
for many, if not all, pseudo-Anosov $f$,
and there is a uniform gap that depends only on $f$.

For any pseudo-Anosov $(S,f)$,
one can always find a sequence $(S'_n, f'_n)_{n\in\Natural}$
as in Theorem \ref{main_vhev_WP},
such that $\mathbf{h}(P_n)>0$ holds for each $f'_n$; see \cite{Liu_vhsr}.
Hence, $\mathbf{m}(P_n)>0$ and $\mathbf{w}(P_n)>0$ also hold.

Note that $\varlimsup_{n}\mathbf{w}(P_n)/[S'_n:S]>0$ holds 
if and only if $\varlimsup_{n}\mathbf{m}(P_n)/[S'_n:S]>0$ holds.
Although it appears heuristically very possible, 
there are no known examples of a pseudo-Anosov $(S,f)$
and a sequence $(S'_n, f'_n)_{n\in\Natural}$
as in Theorem \ref{main_vhev_WP}, such that $\mathbf{w}(P_n)$ or $\mathbf{m}(P_n)$
grows strictly linearly fast compared to the covering degree.

Below, we outline the proof of Theorem \ref{main_vhev_WP}.
Our argument is inspired by \cite{McMullen_entropy}.
The main difference is that
we make use of invariant metrics on the Teichm{\"u}ller space 
related to the Weil--Petersson metric, rather than the Teichm{\"u}ller metric.

Using the Bergman metric on Riemann surfaces $X$ marked by $S$,
one obtains an invariant Riemannian metric on the Teichm\"{u}ller space $\mathrm{Teich}(S)$,
which we call the Habermann--Jost metric.
The way is just the same
as obtaining the Weil--Petersson metric on $\mathrm{Teich}(S)$
from the conformal hyperbolic metric on $X$.
Using the natural map $\mathrm{Teich}(S)\to \mathfrak{H}(S)$ 
of the Teichm\"{u}ller space to the Siegel space of Hodge structures on $H_1(S;\Complex)$,
one obtains an invariant Riemannian pseudometric on $\mathrm{Teich}(S)$
pulling back the Siegel metric on $\mathfrak{H}(S)$,
which we call the Royden--Siegel metric.
Upon suitable normalization,
the $L^2$ norms of these metrics satisfy the comparison
$2\times\|\xi\|_{\mathtt{HJ}}\geq\|\xi\|_{\mathtt{RS}}$,
for any tangent vector $\xi\in T_X\mathrm{Teich}(S)$.

Moreover, virtual versions of the above metrics on $\mathrm{Teich}(S)$ 
can be defined using regular finite covering $S'\to S$
and the natural embedding $\mathrm{Teich}(S)\to \mathrm{Teich}(S')$.
Upon suitable normalization, 
the $L^2$ norms of the virtual Habermann--Jost metrics satisfy
the convergence
$\lim_{n}\|\xi\|_{\mathtt{HJ}'_n}=\|\xi\|_{\mathtt{WP}}/\sqrt{4\pi}$,
for any sequence of covers $(S'_n)_{n\in\Natural}$ as assumed.
This is basically because the virtual Bergmann metrics $g_{\mathtt{B}'_n}$ on $X$
converges to the rescaled conformal hyperbolic metric $g_{\mathtt{h}}/4\pi$.

With the above facts, 
consider any point on the Weil--Petersson axis in $\mathrm{Teich}(S)$
of a given pseudo-Anosov $f$ on $S$. 
Then, essentially speaking, we can obtain comparison of translation lengths
(see Definition \ref{ell_X_def}) as follows:
$$\frac{\tlen_{\mathtt{WP}}(f)}{\sqrt{4\pi}}
\geq \varlimsup_{n\to\infty}\, \tlen_{\mathtt{HJ}'_n}(f)
\geq \varlimsup_{n\to\infty}\, \frac{\tlen_{\mathtt{RS}'_n}(f)}2
\geq \varlimsup_{n\to\infty}\, \frac{\tlen_{\mathtt{S}}(f'_{n*})}{2\times\sqrt{[S'_n:S]}}.$$
The nominator $\tlen_{\mathtt{S}}(f'_{n*})$ in the last expression 
refers to the Siegel translation length of $f'_{n*}\in\mathrm{Sp}(H_1(S'_n;\Real))$ 
acting on $\mathfrak{H}(S'_n)$.
On the other hand,
we can establish the formula
$$\tlen_{\mathtt{S}}(f'_{n*})=2\times \sqrt{\mathbf{w}(P_n)},$$
by studying the Siegel geometry on the generalized upper half plane model.
Then the desired inequality follows.

This paper is organized as follows.
In Section \ref{Sec-prelim}, 
we recall structures associated to the Teichm\"uller space,
and general properties of translation length.
In Section \ref{Sec-Siegel_tl},
we determine the Siegel translation length of a symplectic linear transformation.
In Section \ref{Sec-metrics},
we recall several invariant Riemannian metrics on the Teichm\"uller space, as mentioned above,
and discuss their virtual versions.
In Section \ref{Sec-two_lemmas},
we prove two technical lemmas regarding comparison and convergence of virtual metrics.
In Section \ref{Sec-main_proof},
we prove our main result (Theorem \ref{main_vhev_WP}).

\section{Preliminaries}\label{Sec-prelim}
In this section, we recall background materials needed for our discussion.
For general reference books, 
see \cite{Hubbard_book} for Teichm{\"u}ller theory,
and \cite{Siegel_book} for the Siegel geometry,
and \cite[Chapter II.6]{Bridson--Haefliger_book} 
about the translation length of isometries on metric spaces.
We only consider closed surfaces of hyperbolic type,
as it suffices for our purpose.

\subsection{The Teichm{\"u}ller space}
Let $S$ be an oriented closed surface of genus at least $2$.
The \emph{Teichm{\"u}ller space} $\mathrm{Teich}(S)$ 
consists of all the isotopy classes of complex structures on $S$,
compatible with the fixed orientation.
Equivalently, we think of any point of $\mathrm{Teich}(S)$ 
as represented by 
a Riemann surface $X$ together 
with an orientation-preserving homeomorphism $S\to X$,
called the \emph{marking} of $X$,
and denote the point as $X$ with the marking implicit.

We think of $\mathrm{Teich}(S)$ as a smooth (real) manifold, 
diffeomorphic to an open cell of dimension $6\times(\mathrm{genus}(S)-1)$.
There are natural identifications of the tangent and the cotangent spaces (as real vector spaces)
\begin{eqnarray*}
T^*_X\mathrm{Teich}(S)&\cong& Q(X),\\
T_X\mathrm{Teich}(S)&\cong& B(X)/Q(X)^{\perp},
\end{eqnarray*}
at any point $X\in \mathrm{Teich}(S)$.
Here,
$Q(X)$ denotes the space of all the holomorphic quadratic differentials on $X$,
and $B(X)$ denotes the space of all the $L^\infty$ Beltrami differentials on $X$.
These spaces pair naturally as
$$Q(X)\times B(X) \longrightarrow  \Real\colon
(q,\mu) \mapsto \int_X q\mu,$$
so $Q(X)^\perp$ refers to the subspace of $B(X)$ annihilated by $Q(X)$,
namely,
$\mu\in B(X)$ lies in $Q(X)^\perp$ if and only if $\int_X q\mu=0$ holds for all $q\in Q(X)$.

The \emph{mapping class group} $\mathrm{Mod}(S)$ of $S$,
consisting of all the isotopy classes of orientation-preserving self-homeomorphisms,
acts properly and diffeomorphically on $\mathrm{Teich}(S)$,
transforming the markings of points.

There are many natural differential metrics (or pseudometrics) on $\mathrm{Teich}(S)$
that are invariant under the action of $\mathrm{Mod}(S)$.
For example,
the \emph{Teichm{\"u}ller metric} $d_{\mathtt{T}}$
is an invariant Finsler metric determined uniquely by the formula
$$d_{\mathtt{T}}(X,Y)=\frac{\log K(X,Y)}2,$$
for any $X,Y\in\mathrm{Teich}(S)$,
where $K(X,Y)>1$ denotes the quasiconformity constant of the Teichm\"uller extremal map $X\to Y$
that commutes homotopically with the markings.
The infinitesmal forms of $d_{\mathtt{T}}$ are the norms
$$\|q\|_{\mathtt{T}}=\int_X |q|$$
for any $q$ in $Q(X)\cong T^*_X\mathrm{Teich}(S)$, and
$$\|\xi\|_{\mathtt{T}}=\sup\left\{\left|\int_X q\mu\right|\colon \int_X |q|\leq 1\right\}$$
for any $\xi=\mu+Q(X)^\perp$ in $B(X)/Q(X)^\perp\cong T_X\mathrm{Teich}(S)$.

We recall the Weil--Petersson metric $d_{\mathtt{WP}}$
and other invariant Riemannian metrics on $\mathrm{Teich}(S)$
in Section \ref{Sec-metrics} with more discussion.

\subsection{The Siegel space}
Let $V$ be a real vector space of dimension $2p$
furnished with a symplectic form $\omega$.
Then $\omega$ extends complex linearly over
the complex vector space $V\otimes_\Real\Complex$
as a complex-valued alternating $2$-form,
and determines a Hermitian form of signature $(p,p)$ on $V\otimes_\Real\Complex$,
$$\langle z,w\rangle=\frac{\sqrt{-1}}2\cdot\omega(z,\bar{w}).$$
The \emph{Siegel space} $\mathfrak{H}(V)$ consists of 
all the Hermitian orthogonal splittings
$$V\otimes_\Real\Complex=V^{1,0}\oplus V^{0,1},$$
such that $\langle\_,\_\rangle$ is positive definite
on $V^{1,0}$, and the complex conjugation $z\mapsto \bar{z}$
maps $V^{1,0}$ isomorphically onto $V^{0,1}$.
We call any point of the Siegel space $\mathfrak{H}(V)$
a \emph{Hodge structure} 
on $V\otimes_\Real\Complex$ with respect to $\omega$.

The group $\mathrm{Sp}(V)\cong \mathrm{Sp}(2p,\Real)$ 
of symplectic linear transformations acts transitively on $\mathfrak{H}(V)$,
such that any $\varphi\in\mathrm{Sp}(V)$ extends complex linearly over $V\otimes_\Real\Complex$,
and takes any Hodge structure $V^{1,0}\oplus V^{0,1}$ to 
$\varphi(V^{1,0})\oplus\varphi(V^{0,1})$.
The isotropy group at $V^{1,0}\oplus V^{0,1}$ is isomorphic to 
the unitary group $\mathrm{U}(V^{1,0})\cong \mathrm{U}(p)$, 
acting simultaneously on the summands $V^{1,0}$ and $V^{0,1}$
as canonical unitary transformations and their complex conjugates, respectively.
In particular, $\mathfrak{H}(V)$ is 
the homogenous space associated to $\mathrm{Sp}(V)$,
diffeomorphic to an open cell of dimension $p^2+p$.

The most important example to our interest 
is the Siegel space of Hodge structures on a surface.
If $S$ is an oriented closed surface of genus at least $2$,
then the first real cohomology $H^1(S;\Real)$ is furnished with
a natural symplectic form,
evaluating the cup product of any pair of $1$-classes 
on the fundamental class.
In this case, we simply denote the Siegel space of Hodge structures
on $H^1(S;\Complex)$ as $\mathfrak{H}(S)$.
For any Teichm{\"u}ller point $X\in \mathrm{Teich}(S)$,
we obtain a unique Hermitian orthogonal splitting
$$H^1(S;\Complex)\cong H^1(X;\Complex)=H^{1,0}(X)\oplus H^{0,1}(X)\cong \Omega^1(X)\oplus \overline{\Omega^1(X)},$$
where the complex linear subspaces $H^{1,0}(X)$ and $H^{0,1}(X)$ 
of $H^1(X;\Complex)$ are naturally identified as the spaces $\Omega^1(X)$ and $\overline{\Omega^1(X)}$
of the holomorphic and the anti-holomorphic differentials on $X$, respectively,
and 
where $H^1(S;\Complex)$ is identified with $H^1(X;\Complex)$ via the marking.
Therefore, the construction determines a natural map
\begin{equation}\label{T_to_H}
J\colon \mathrm{Teich}(S)\to \mathfrak{H}(S),
\end{equation}
which is equivariant with respect to the natural group homomorphism
$\mathrm{Mod}(S)\to\mathrm{Sp}(H^1(S;\Real))$.
It is known that $J$ is surjective and smooth,
and the tangent map of $J$ is injective except along the hyperelliptic locus;
see \cite[Section 3]{Royden_invariant_metric} for more information.

Back to the general setting with $(V,\omega)$.
There is a unique invariant Riemannian metric on 
$\mathfrak{H}(V)$ with respect to $\mathrm{Sp}(V)$, up to normalization.
In fact, 
$(\mathfrak{H}(V),\mathrm{Sp}(V))$ forms a simple Riemannian symmetric space of noncompact type,
(and morover, a Hermitian symmetric domain).
This metric has been studied by C.~L.~Siegel systematically
in his expository book \cite{Siegel_book}.
Fixing a symplectic basis of $V$,
we can identify $V$ as the real linear space 
spanned by $\vec{x}_1,\cdots,\vec{x}_p,\vec{y}_1,\cdots,\vec{y}_p$ and furnished with 
the standard symplectic form $\vec{x}^*_1\wedge \vec{y}^*_1+\cdots+\vec{x}^*_p\wedge \vec{y}^*_p$.
Then the transformational geometry of 
the Siegel space $(\mathfrak{H}(V),\mathrm{Sp}(V))$ can be identified
with Siegel's generalized upper half plane model,
on which the invariant metric can be described explicitly.
We recall the explicit description below,
following Siegel's normalization in \cite{Siegel_book}.

\subsection{The generalized upper half plane}
For any natural number $p$, the \emph{generalized upper half plane} $\mathfrak{H}_p$ of rank $p$ 
refers to the open subset of symmetric $p\times p$--matrices of complex entries,
such that the imaginary parts are positive definite.
Namely,
\begin{equation}\label{Siegel_UHP}
\mathfrak{H}_p=\left\{ Z\in \mathrm{Sym}_{p\times p}(\Complex)\colon \Im Z>0\right\},
\end{equation}
where we denote
$\mathrm{Sym}_{p\times p}(\Complex)=\{ Z\in\mathrm{Mat}_{p\times p}(\Complex)\colon Z^\dagger=Z\}$
($\dagger$ meaning transpose) and $\Im Z=(Z-\bar{Z})/2\sqrt{-1}$
\cite[Chapter I, \S 2]{Siegel_book}.

The Riemannian metric tensor of the the Siegel metric 
at any point $Z\in \mathfrak{H}_p$ is defined as
\begin{equation}\label{g_S_def}
g_{\mathtt{S}}=\mathrm{tr}\left(Y^{-1}\cdot\ud Z\cdot Y^{-1}\cdot \overline{\ud Z}\right)
\end{equation}
where we denote $Y=\Im Z$ and 
treat the matrix entries of the real and the imaginary parts as coordinates
\cite[Chapter III, \S 11]{Siegel_book}.
The Siegel metric $g_{\mathtt{S}}$ has nonpositive sectional curvature everywhere 
\cite[Chapter III, \S 17]{Siegel_book}.
In other words, $(\mathfrak{H}_p,g_{\mathtt{S}})$ forms a Hadamard manifold,
(by definition, simply connected, metric complete, and nonpositively curved).

For any pair of points $Z,W\in\mathfrak{H}_p$,
the generalized cross-ratio matrix
$$R=R(Z,W)=(Z-W)\cdot(Z-\bar{W})^{-1}\cdot(\bar{Z}-\bar{W})\cdot(\bar{Z}-W)^{-1}$$
is defined in $\mathrm{Mat}_{p\times p}(\Complex)$, with all characteristic roots contained in $[0,1)$.
Hence the series expression
$$\log^2\left(\frac{1+R^{1/2}}{1-R^{1/2}}\right)=4R\cdot\left(\sum_{m=0}^\infty\frac{R^m}{2m+1}\right)^2$$
converges absolutely to a matrix in $\mathrm{Mat}_{p\times p}(\Complex)$.
The Siegel distance between the pair of points $Z,W\in\mathfrak{H}_p$ can be calculated as
\begin{equation}\label{d_S_global}
d_{\mathtt{S}}(Z,W)
=\sqrt{\mathrm{tr}\left(\log^2\left(\frac{1+R^{1/2}}{1-R^{1/2}}\right)\right)}.
\end{equation}
See \cite[Chapter III, \S 13]{Siegel_book}.

The symplectic group $\mathrm{Sp}(2p,\Real)$ acts on $\mathfrak{H}_p$
by generalized fractional linear transformations, as follows.
Any matrix $\varphi\in\mathrm{Sp}(2p,\Real)$ can be divided into four blocks
$A,B,C,D\in\mathrm{Mat}_{p\times p}(\Real)$,
such that the basis transforms as
$$\varphi_*(\begin{array}{cccccc}\vec{x}_1&\cdots&\vec{x}_p&\vec{y}_1&\cdots&\vec{y}_p\end{array})=
(\begin{array}{cccccc}\vec{x}_1&\cdots&\vec{x}_p&\vec{y}_1&\cdots&\vec{y}_p\end{array})
\left(\begin{array}{cc} A& B\\ C& D\end{array}\right).$$
Then $\varphi$ being symplectic is equivalent to the conditions
$A^\dagger C=C^\dagger A$, and $D^\dagger B=B^\dagger D$ and
$A^\dagger D-C^\dagger B=I$ altogether
($I$ denoting the identity matrix).
The action of $\mathrm{Sp}(2p,\Real)$ on $\mathfrak{H}_p$
is explicitly
$$\varphi\colon Z\mapsto (AZ+B)\cdot (CZ+D)^{-1}$$
for any $\varphi\in\mathrm{Sp}(2p,\Real)$ and $Z\in\mathfrak{H}_p$.
See \cite[Chapter II, \S\S 4--5]{Siegel_book}.

\begin{remark}
For any real vector space $V$ of dimension $2p$
furnished with a symplectic form $\omega$, 
an identification
$(\mathfrak{H}(V),\mathrm{Sp}(V))\cong(\mathfrak{H}_p,\mathrm{Sp}(2p,\Real))$
is canonically determined by fixing a symplectic basis of $V$.
Since the canonical identifications $V\cong\Real^{2p}$ and $\mathrm{Sp}(V)\cong\mathrm{Sp}(2p,\Real)$ are obvious,
it remains to identify any point $Z\in\mathfrak{H}_p$ with a Hodge structure on $V\otimes_\Real\Complex\cong\Complex^{2p}$.
We first require the distinguished point $Z=\sqrt{-1}\cdot I$ 
to be identified with the Hodge structure $V^{1,0}\oplus V^{0,1}$
where $V^{1,0}$ is the real linear subspace spanned by the complex vectors:
$$\begin{array}{ccc}\vec{z}_1=\vec{x}_1+\sqrt{-1}\cdot \vec{y}_1,&\cdots,&\vec{z}_p=\vec{x}_p+\sqrt{-1}\cdot\vec{y}_p\end{array}$$
Hence, $V^{0,1}=\overline{V^{1,0}}$.
In general, any point $X+\sqrt{-1}\cdot Y$ in $\mathfrak{H}_p$
is the transformation image of $\sqrt{-1}\cdot I$ under the symplectic matrix
$$M=
\left(\begin{array}{cc} I& X\\ 0& I\end{array}\right)\cdot
\left(\begin{array}{cc} Y^{1/2}& 0\\ 0& Y^{-1/2}\end{array}\right)=
\left(\begin{array}{cc} Y^{1/2}& X\cdot Y^{-1/2}\\ 0& Y^{-1/2}\end{array}\right),$$
so the Hodge structure corresponding to $Z=X+\sqrt{-1}\cdot Y$
has $V^{1,0}$ spanned by $M\vec{z}_1,\cdots,M\vec{z}_p$.
\end{remark}

\subsection{Translation length}
\begin{definition}\label{ell_X_def}
For any metric space $({\mathcal{X}},d_{\mathcal{X}})$, the \emph{translation length} of an isometry
$\sigma\colon {\mathcal{X}}\to {\mathcal{X}}$
is defined as
$$\tlen_{\mathcal{X}}(\sigma)=\inf_{x\in {\mathcal{X}}} d_{\mathcal{X}}(x,\sigma.x).$$
\end{definition}

The translation length can be thought of as a function on the isometry group of ${\mathcal{X}}$,
$$\tlen_{\mathcal{X}}\colon \mathrm{Isom}({\mathcal{X}})\to [0,+\infty).$$
This function is clearly invariant under conjugations and inversion.
Moreover, it satisfies the following upper semi-continuity
with respect to sequences of pointwise convergence.

\begin{lemma}\label{tlen_upper_semicontinuous}
	If $(\sigma_m)_{m\in\Natural}$ is a sequence of isometries on 
	a metric space $({\mathcal{X}},d_{\mathcal{X}})$ that converges to an isometry $\varphi_\infty$ everywhere in ${\mathcal{X}}$,
	then
	$$\varlimsup_{m\to\infty} \tlen_{\mathcal{X}}(\sigma_m)\leq \tlen_{\mathcal{X}}(\sigma_\infty).$$
\end{lemma}

\begin{proof}
	For any $x\in {\mathcal{X}}$ and $\epsilon>0$, we obtain
	$d_{\mathcal{X}}(x,\sigma_m.x)\leq d_{\mathcal{X}}(x,\sigma_\infty.x)+\epsilon$ for all sufficiently large $m$.
	Then $\tlen_{\mathcal{X}}(\sigma_m)\leq d_{\mathcal{X}}(x,\sigma_\infty.x)+\epsilon$, and then
	$\varlimsup_{m\to\infty} \tlen_{\mathcal{X}}(\sigma_m)\leq d_{\mathcal{X}}(x,\sigma_\infty.x)+\epsilon$.
	Taking the infimum over all $x\in {\mathcal{X}}$, we obtain
	$\varlimsup_{m\to\infty} \tlen_{\mathcal{X}}(\sigma_m)\leq \tlen_{\mathcal{X}}(\sigma_\infty)+\epsilon$.
	Let $\epsilon$ tend to $0+$, the asserted inequality follows.
\end{proof}

\section{The Siegel translation length of symplectic linear transformations}\label{Sec-Siegel_tl}
In this section, we figure out the translation length of a symplectic matrix
with respect to its fractional linear transformation on the generalized upper half plane
and the Siegel metric.
This is not a difficult task, 
the most importantly because symplectic linear transformations
have been well understood. 
In fact, conjugacy classes in $\mathrm{Sp}(2p,\Real)$ have been classified
with normal forms all listed. The theory is similar to 
the well-known Jordan normal form theory for $\mathrm{SL}(n,\Real)$,
although notationally more involved.
For details of that theory, we refer to J.~Gutt \cite{Gutt_normal_form}.
However, 
we need a few ingredients from \cite{Gutt_normal_form} 
in our proof of Theorem \ref{tlen_Siegel},
and we elaborate when we use them.

\begin{theorem}\label{tlen_Siegel}
	Let $(V,\omega)$ be a finite-dimensional real symplectic vector space.
	Then,	for any symplectic linear transformation
	$\varphi\in\mathrm{Sp}(V)$,
	the Siegel translation length of $\varphi$ satisfies the formula
	$$\tlen_{\mathtt{S}}(\varphi)=2\times\sqrt{\mathbf{w}(P_\varphi)}$$
	where $P_\varphi$ denotes the characteristic polynomial of $\varphi$,
	and where $\mathbf{w}$ denotes the Jensen square sum as defined in (\ref{w_def}).
\end{theorem}

The rest of this section is devoted to the proof of Theorem \ref{tlen_Siegel}.

\begin{lemma}\label{Siegel_Pythagorean}
	If $V=V_1\oplus\cdots\oplus V_k$ is a symplectic orthogonal decomposition
	of $V$ into symplectic summands $V_i$ that are invariant under $\varphi$.
	Then
	$$\tlen_{\mathtt{S}}(\varphi)=\sqrt{\tlen_{\mathtt{S}}(\varphi_1)^2+\cdots+\tlen_{\mathtt{S}}(\varphi_k)^2}$$
	where $\varphi_i\in\mathrm{Sp}(V_i)$ denotes the restricted transformation of $\varphi$.
\end{lemma}

\begin{proof}
	The decomposition induces a natural inclusion
	$\mathfrak{H}(V_1)\times\cdots\times\mathfrak{H}(V_k)$
	into $\mathfrak{H}(V)$.
	With respect to the cartesian product of Siegel metrics on
	$\mathfrak{H}(V_1)\times\cdots\times\mathfrak{H}(V_k)$
	and the Siegel metric on $\mathfrak{H}(V)$,
	the inclusion is an isometric embedding
	with totally geodesic image,
	by (\ref{g_S_def}) and simple observation.
	Since the Siegel space $\mathfrak{H}(V)$ is Hadamard with respect to the Siegel metric,
	the closest point projection  
	$\mathfrak{H}(V)\to \mathfrak{H}(V_1)\times\cdots\times\mathfrak{H}(V_k)$
	is well-defined and distance non-increasing;
	see \cite[Chapter II, Corollary 2.5]{Bridson--Haefliger_book}.
	It follows that $\tlen_{\mathtt{S}}(\varphi)$ 
	is witnessed by some sequence of points 
	in $\mathfrak{H}(V_1)\times\cdots\times\mathfrak{H}(V_k)$,
	so it agrees with the translation distance of $\varphi$
	acting on $\mathfrak{H}(V_1)\times\cdots\times\mathfrak{H}(V_k)$.
	The latter is evidently equal to
	$(\tlen_{\mathtt{S}}(\varphi_1)^2+\cdots+\tlen_{\mathtt{S}}(\varphi_k)^2)^{1/2}$.
\end{proof}

Let $(V,\omega)$ be a finite-dimensional real symplectic vector space.
For any element $\varphi\in\mathrm{Sp}(V)$, 
it follows from the Jordan--Chevalley decomposition
that there is a unique, commutative factorization of $\varphi$ in $\mathrm{Sp}(V)$,
$$\varphi=\varphi_{\mathtt{ss}}\varphi_{\mathtt{u}}=\varphi_{\mathtt{u}}\varphi_{\mathtt{ss}},$$
such that the eigenvectors of $\varphi_{\mathtt{ss}}$ span $V\otimes_\Real\Complex$,
and the only eigenvalue of $\varphi_{\mathtt{u}}$ is $1$.
In fact, the Jordan--Chevalley decomposition in
the semisimple real Lie algebra $\mathfrak{sp}(V)$ is the same as
inherited from $\mathfrak{gl}(V)$, and it detemines the factorization in 
the linear group $\mathrm{Sp}(V)$; see \cite[Chapter II, \S 6.4]{Humphreys_book}.

We refer to $\varphi_{\mathtt{ss}}$ and $\varphi_{\mathtt{u}}$
as the \emph{semisimple} factor and the \emph{unipotent} factor of $\varphi$,
respectively.

\begin{lemma}\label{semisimple_tlen_Siegel}
	Theorem \ref{tlen_Siegel} holds if $\varphi=\varphi_{\mathtt{ss}}$.
\end{lemma}

\begin{proof}
	We say that $(V,\varphi)$ is \emph{symplectic simple} if $\varphi=\varphi_{\mathtt{ss}}$,
	and if $V$ contains no nontrivial $\varphi$--invariant symplectic subspaces.
	In this case, $(V,\varphi)$ decomposes as a direct sum
	of symplectic simple transformations $(V_i,\varphi_i)$.
	In view of the Pythagorean formula (Lemma \ref{Siegel_Pythagorean}) 
	for $\tlen_{\mathtt{S}}(\varphi)$ 
	and the obvious similar formula	for $2\times\sqrt{\mathbf{w}(P_{\varphi})}$,
	it suffices to prove the lemma for any symplectic simple transformation.

	Any symplectic simple transformation $(V,\varphi)$ 
	is either $2$--dimensional central (with a unique eigenvalue $\lambda=\pm1$ of multiplicity $2$),
	or $2$--dimensional hyperbolic (with distinct real eigenvalues $\{\lambda,1/\lambda\}$),
	or $2$--dimensional elliptic (with distinct unimodular eigenvalues $\{\lambda,\bar\lambda\}$),
	or otherwise, $4$--dimensional `loxodromic'
	(with distinct complex eigenvalues $\{\lambda,\bar\lambda,1/\lambda,1/\bar\lambda\}$
	off the unit circle and the real axis).
	See \cite[Theorem 1.2]{Gutt_normal_form} for general normal forms
	of matrices in $\mathrm{Sp}(2p,\Real)$.
	
	When $V$ is $2$--dimensional,
	the formula for	$\tlen_{\mathtt{S}}(\varphi)$ is actually well-known.
	In this case, the Siegel metric on Siegel space $\mathfrak{H}(V)$ agrees
	with the hyperbolic metric on the upper-half complex plane
	(of constant curvature $-1$),
	and $\mathrm{Sp}(V)\cong\mathrm{SL}(2,\Real)$ acts the same way
	as the fractional linear transformations,
	so 
	$$\tlen_{\mathtt{S}}(\varphi)=
	\begin{cases}
	0& \varphi\mbox{ central/elliptic}
	\\
	|2\log|\lambda||& \varphi\mbox{ hyperbolic}
	\end{cases}.$$
	In any of the subcases, $\tlen_{\mathtt{S}}(\varphi)$ is equal to
	$2\times\sqrt{\mathbf{w}(P_\varphi)}$, as asserted.
	
	When $V$ is $4$--dimensional, 
	the Siegel space $\mathfrak{H}(V)$ can be identified
	with the generalized upper half plane $\mathfrak{H}_2$,
	consisting of all the symmetric complex $2\times2$--matrices
	with positive definite imaginary part.
	Any symplectic transformation 
	with distinct eigenvalues $\{\lambda,\bar\lambda,1/\lambda,1/\bar\lambda\}$
	can be conjugated in $\mathrm{Sp}(4,\Real)$
	into the normal form
	$$
	\varphi=\left(\begin{array}{cc} A&B\\C&D\end{array}\right)
	=
	\left(\begin{array}{cc} 
	\rho\cdot\left({\begin{array}{cc} \cos\theta & -\sin\theta \\ \sin\theta & \cos\theta\end{array}}\right) & \\
	& \rho^{-1}\cdot\left({\begin{array}{cc} \cos\theta & -\sin\theta \\ \sin\theta & \cos\theta\end{array}}\right)
	\end{array}\right)
	$$
	denoting $\rho=|\lambda|$ and $\theta=\arg \lambda$.
	Recall that $\varphi$ acts isometrically on $\mathfrak{H}_2$
	as $Z\mapsto (AZ+B)(CZ+D)^{-1}$.
	The purely positive imaginary scalar $2\times2$--matrices form 
	a $\varphi$--invariant geodesic $l$
	in the Hadamard manifold $\mathfrak{H}_2$.
	As the closest point projection $\mathfrak{H}_2\to l$ 
	is distance non-increasing 
	\cite[Chapter II, Corollary 2.5]{Bridson--Haefliger_book},
	the translation length of $\varphi$ is realized at every point on $l$.
	On $l$ we can apply (\ref{d_S_global}) and compute directly
	$$\tlen_{\mathtt{S}}(\varphi)=\sqrt{2}\cdot \left|2\log |\lambda|\right|,$$
	which is also $2\times \sqrt{\mathbf{w}(P_\varphi)}$, as asserted.
\end{proof}

\begin{lemma}\label{semisimple_tlen_suffices} 
$$\tlen_{\mathtt{S}}(\varphi)=\tlen_{\mathtt{S}}(\varphi_{\mathtt{ss}})$$
\end{lemma}

\begin{proof}
	For any $r\in\Real$,
	the element	$\varphi_{\mathtt{u}}^r\in\mathrm{Sp}(V)$ 
	is well defined via the series expansion 
	$(1+x)^r=\sum_{j=0}^\infty{r\choose j}x^j$,
	because $\varphi_{\mathtt{u}}-\mathbf{1}$ is nilpotent in $\mathrm{End}_\Real(V)$.
	Moreover, for any $r>0$,	the element
	$\varphi_{\mathtt{ss}}\varphi_{\mathtt{u}}^r$ is conjugate 
	to $\varphi=\varphi_{\mathtt{ss}}\varphi_{\mathtt{u}}$ in $\mathrm{Sp}(V)$.
	For example, this follows from Gutt's characterization of conjugacy classes;
	the characterization only uses eigenvalues, and dimensions of generalized eigenspaces,
	and the ranks and signatures of certain symmetric forms on the generalized eigenspaces,
	all of which are obviously preserved as $r>0$ varies;
	see \cite[Theorem 1.2]{Gutt_normal_form}.
	Observe $\tlen_{\mathtt{S}}$ is constant on any conjugacy class of $\mathrm{Sp}(V)$.
	By Lemma \ref{tlen_upper_semicontinuous},
	we obtain
	$$\tlen_{\mathtt{S}}(\varphi)=
	\lim_{m\to\infty}\tlen_{\mathtt{S}}\left(\varphi_{\mathtt{ss}}\varphi_{\mathtt{u}}^{1/m}\right)\leq
	\tlen_{\mathtt{S}}(\varphi_{\mathtt{ss}}).$$
	
	To see the other direction, 
	we observe that symplectic linear transformations with no redundant characteristic roots
	form a dense open subset of $\mathrm{Sp}(V)$.
	(In fact, the discriminant of the characteristic polynomial being zero
	determines an algebraic, proper subset of the real algebraic group $\mathrm{Sp}(V)$,
	whose interior has to be empty.)
	Therefore, we can take a sequence $(\psi_m)_{m\in\Natural}$ in $\mathrm{Sp}(V)$,
	converging to $\varphi$, such that each $\psi_m$ has no redudant characteristic roots.
	In particular, $\psi_m$ are all semisimple.
	It follows from Lemma \ref{semisimple_tlen_Siegel}
	that the formula in Theorem \ref{tlen_Siegel} holds for all $\tlen_{\mathtt{S}}(\psi_m)$
	and for $\tlen_{\mathtt{S}}(\varphi_{\mathtt{ss}})$.
	Note that $\varphi_{\mathtt{ss}}$ has the same characteristic polynomial as that of $\varphi$,
	so the characteristic roots of $\psi_m$ converge to those of $\varphi_{\mathtt{ss}}$.
	By Lemma \ref{tlen_upper_semicontinuous}, we obtain
	$$\tlen_{\mathtt{S}}(\varphi_{\mathtt{ss}})=
	\lim_{m\to\infty} \tlen_{\mathtt{S}}(\psi_m) \leq	
	\tlen_{\mathtt{S}}(\varphi).$$ 
	
	We conclude that $\tlen_{\mathtt{S}}(\varphi)$ must be equal to $\tlen_{\mathtt{S}}(\varphi_{\mathtt{ss}})$.
\end{proof}

Theorem \ref{tlen_Siegel} follows from Lemmas \ref{semisimple_tlen_Siegel} and \ref{semisimple_tlen_suffices}.

\section{Invariant Riemannian metrics on the Teichm{\"u}ller space}\label{Sec-metrics}
In this section, we recall some invariant Riemannian metrics or pseudometrics
on the Teichm{\"u}ller space, and also consider similar constructions
associated to regular finite covers of the surface. 
As some of the metrics have name collision in the literature,
we rename a few in this paper, not necessarily suggesting an accurate attribution.

\subsection{Ordinary versions}

\subsubsection{The Weil--Petersson metric}
By the uniformization theorem, 
there is a unique conformal hyperbolic metric $g_{\mathtt{h}}$ on $X$,
(that is, a Riemannian metric of constant curvature $-1$ 
conformal to the complex charts).
The well-known Gauss--Bonnet theorem implies
$$\mathrm{Area}_{\mathtt{h}}(X)=-2\pi\chi(S)=4\pi\times (\mathrm{genus}(S)-1).$$
%
%
%
With respect to $g_{\mathtt{h}}$,
any tangent vector $\xi\in T_X\mathrm{Teich}(S)$ 
is represented by a unique harmonic Beltrami differential $\mu\in B(X)$,
that is, $\mu=q/g_{\mathtt{h}}$ for some 
holomorphic quadratic differential $q\in Q(X)$.
The \emph{Weil--Petersson metric} on $\mathrm{Teich}(S)$
is the Riemannian metric determined by
\begin{equation}\label{WP_def}
\|\xi\|_{\mathtt{WP}}^2=\int_X |\mu|^2\, \ud\mathrm{Area}_{\mathtt{h}}= \int_X \frac{q\bar{q}}{g_{\mathtt{h}}}
\end{equation}
for any $\xi\in T_X\mathrm{Teich}(S)=B(X)/Q(X)^\perp$
and the hyperbolic harmonic representative $\mu=\overline{q/g_{\mathtt{h}}}$.

The Weil--Petersson metric on $\mathrm{Teich}(S)$ has nonpositive sectional curvature.
It is not complete, but any pair of points in $\mathrm{Teich}(S)$ 
can be connected with a unique geodesic segment.
The mapping class group $\mathrm{Mod}(S)$ acting on $\mathrm{Teich}(S)$
preserves the Weil--Petersson metric.
Moreover, any pseudo-Anosov mapping class $f\in\mathrm{Mod}(S)$ has an axis,
namely, a unique invariant complete geodesic, 
along which the Weil--Petersson translation length $\tlen_{\mathtt{WP}}(f)$ is realized.
See the survey \cite{Wolpert_wp} of S.~A.~Wolpert for more information.

\subsubsection{The Habermann--Jost metric}
Using the Bergmann metric instead of the hyperbolic metric,
we obtain what we call the Habermann--Jost metric on $\mathrm{Teich}(S)$.
This metric has been studied 
by L.~Habermann and J.~Jost \cite{Habermann--Jost} in detail.

For any $X\in\mathrm{Teich}(S)$,
the \emph{Bergmann metric} $g_{\mathtt{B}}$ on $X$ refers to the pullback
of the canonical flat metric on $\mathrm{Jac}(X)=\Omega^1(X)^*/H_1(X;\Integral)$
via the Abel--Jacobi map $X\to \mathrm{Jac}(X)$,
which is independent of the auxiliary base point.
To avoid confusion about the normalization, 
one makes sure
$$\mathrm{Area}_{\mathtt{B}}(X)=\mathrm{genus}(S).$$
The \emph{Haberman--Jost metric} on $\mathrm{Teich}(S)$
is the Riemannian metric determined by
\begin{equation}\label{HJ_def}
\|\xi\|_{\mathtt{HJ}}^2=\int_X |\mu|^2\,\ud\mathrm{Area}_{\mathtt{B}}
= \int_X \frac{q\bar{q}}{g_{\mathtt{B}}},
\end{equation}
for any $\xi\in T_X\mathrm{Teich}(S)=B(X)/Q(X)^\perp$
and its Bergman harmonic representative $\mu=\overline{q/g_{\mathtt{B}}}\in B(X)$,
for some unique $q\in Q(X)$.

\begin{lemma}\label{HJ_theta}
	Let $\theta_1,\cdots,\theta_p$ be a Hermitian orthonormal basis of $\Omega^1(X)$,
	where $p=\dim_\Complex \Omega^1(X)$ is equal to the genus of $S$.
	Then, the following formulas hold,
		$$g_{\mathtt{B}}=\sum_{j=1}^p \theta_j\bar\theta_j,$$
	and for any $\xi\in T_X\mathrm{Teich}(S)$ represented uniquely as $\mu=\overline{q/g_{\mathtt{B}}}$,
		$$\|\xi\|_{\mathtt{HJ}}^2=\sum_{j=1}^p \int_X \mu\theta_j\,\overline{\mu\theta_j}.$$
\end{lemma}

\begin{proof}
In fact, the formula for $g_{\mathtt{B}}$ in Lemma \ref{HJ_theta} is 
a usual definition of the Bergman metric on $X$, 
and its geometric interpretation is 
our above description with the Abel--Jacobi map.
The formula for $\|\_\|_{\mathtt{HJ}}$ in Lemma \ref{HJ_theta}
follows immediately
from the formula for $g_{\mathtt{B}}$ and the defining expression in (\ref{HJ_def}).
Note that $\overline{\mu\theta_j}$ are differentials of type $(1,0)$,
but not holomorphic in general.
\end{proof}

\subsubsection{The Royden--Siegel metric}
The natural map $J\colon \mathrm{Teich}(S)\to\mathfrak{H}(S)$ 
and the Siegel metric on $\mathfrak{H}(S)$ induces a canonical Riemannian pseudometric
on $\mathrm{Teich}(S)$, 
degenerating exactly on the hyperelliptic locus.
This construction results in what we call the Royden--Siegel metric.
In the survey \cite{Royden_invariant_metric},
H.~L.~Royden discusses several such pullback metrics on $\mathrm{Teich}(S)$
using canonical differential metrics on $\mathfrak{H}(S)$.

For any $\xi\in T_X\mathrm{Teich}(S)$,
the \emph{Royden--Siegel metric} is the Riemannian pseudometric determined by
\begin{equation}\label{RS_def}
\|\xi\|_{\mathtt{RS}}^2=\|J_*(\xi)\|_{\mathtt{S}}^2,
\end{equation}
according to our notations in (\ref{T_to_H}) and (\ref{g_S_def}).

\begin{lemma}\label{RS_theta}
	Let $\theta_1,\cdots,\theta_p$ be a Hermitian orthonormal basis of $\Omega^1(X)$,
	where $p=\dim_\Complex \Omega^1(X)$ is equal to the genus of $S$.
	Then, for any $\xi\in T_X\mathrm{Teich}(S)$ represented by any $\mu\in B(X)$,
	the following formula holds,
	$$\|\xi\|_{\mathtt{RS}}^2=4 \times \sum_{j=1}^p\sum_{k=1}^p \,\left|\int_X \mu\theta_j \theta_k\right|^2.$$
\end{lemma}

\begin{proof}
	The formula is equivalent to \cite[Theorem 2]{Royden_invariant_metric},
	except a factor $4$ due to normalization. 
	To be precise, Royden uses a rescaled metric on the generalized upper half plane that is 
	half of Siegel's in length, 
	(compare \cite[{Eq.~(15) on p.~397}]{Royden_invariant_metric}
	and \cite[{Eq.~(2) on p.~3}]{Siegel_book}).
	With our notations (following Siegel's normalization),
	\cite[Theorem 2]{Royden_invariant_metric} can be rewritten as
	$$\|\xi\|_{\mathtt{RS}}^2=4 \times \int_X\int_X \mu(x')K(x',x'')\overline{\mu(x'')},$$
	where $x'$ parametrizes the first $X$ and $x''$ parametrizes the second $X$.
	The kernel $K(x',x'')$ is 
	a holomorphic quadratic differential in the first variable,
	and an anti-holomorphic quadratic differential in the second variable,
	and is Hermitian symmetric under switching of the variables.
	It is explicitly constructed as
	$$K(x',x'')=\left(\sum_l\theta_l(x')\overline{\theta_l(x'')}\right)^2
	=\sum_j\sum_k \theta_j(x')\theta_k(x')\overline{\theta_j(x'')\theta_k(x'')},$$
	indices all ranging over $\{1,\cdots,p\}$.
	Plug the last expression into the above double integral,
	and move the summations out of the integrations.
	Then the variables separate, and the asserted formula follows.
\end{proof}

\subsection{Virtual versions}
Any regular finite covering map $S'\to S$ induces a canonical embedding
$$\mathrm{Teich}(S)\to\mathrm{Teich}(S'),$$
such that any Riemann surface with marking $S\to X$ 
goes to $S'\to X'$
where $X'\to X$ is holomorphic and equivariant with $S'\to S$. 

The Royden--Siegel metric on $\mathrm{Teich}(S')$ induces 
a pseudometric on $\mathrm{Teich}(S)$ by restriction.
To normalize, 
we introduce the \emph{virtual Royden--Siegel metric}
on $\mathrm{Teich}(S)$ with respect to the regular finite cover $S'$
as the Riemannian pseudometric, such that for any $\xi\in T_X\mathrm{Teich}(S)$, 
we define
\begin{equation}\label{virtual_RS_def}
\|\xi\|^2_{\mathtt{RS}'}=\|\xi'\|^2_{\mathtt{RS}}\times \frac{1}{[S':S]}
=\|J'_*(\xi')\|^2_{\mathtt{S}}\times \frac{1}{[S':S]},
\end{equation}
where $\xi'\in T_{X'}\mathrm{Teich}(S')$ denotes the tangent map image of $\xi$,
and 
where $J'$ denotes the natural map $\mathrm{Teich}(S')\to \mathfrak{H}(S')$ 
as in (\ref{T_to_H}).

We introduce the \emph{virtual Haberman--Jost metric} on $\mathrm{Teich}(S)$
with respect to any regular finite cover $S'$ as follows.
For any $X\in\mathrm{Teich}(S)$,
since the deck transformations on 
the corresponding $X'$ are all biholomorphic,
they all preserve the Bergman metric $g'_{\mathtt{B}}$ on $X'$.
Therefore, $g'_{\mathtt{B}}$ is the pullback of a unique conformal metric $g_{\mathtt{B}'}$ on $X$.
We notice the relation
$$\mathrm{Area}_{\mathtt{B}'}(X)=\frac{\mathrm{Area}_{\mathtt{B}}(X')}{[S':S]}=
\mathrm{genus}(S)-1+\frac{1}{[S':S]}.$$
For any $\xi\in T_X\mathrm{Teich}(S)$, we define
$$\|\xi\|_{\mathtt{HJ}'}^2=
\int_X |\mu|^2\,\ud\mathrm{Area}_{\mathtt{B}'}
= \int_X \frac{q\bar{q}}{g_{\mathtt{B}'}},$$
where $\mu=\overline{q/g_{\mathtt{B}'}}$ 
denotes the Beltrami differential representative of $\xi$
that is harmonic with respect to $g_{\mathtt{B}'}$.
The last expression implies a relation
$$\|\xi\|_{\mathtt{HJ}'}^2=\|\xi'\|_{\mathtt{HJ}}^2\times \frac{1}{[S':S]},$$
where $\xi'\in T_{X'}\mathrm{Teich}(S')$ denotes the tangent map image of $\xi$.

There is no need to introduce any ``virtual Weil--Petersson metric'',
as it would be the same thing as the Weil--Petersson metric.
This is because the conformal hyperbolic metric $g_{\mathtt{h}}'$ on $X'$
agrees with the pullback of $g_{\mathtt{h}}$ on $X$.
We only mention the obvious relation
$$\|\xi\|_{\mathtt{WP}}^2=\|\xi'\|_{\mathtt{WP}}^2\times \frac{1}{[S':S]},$$
where $\xi'\in T_{X'}\mathrm{Teich}(S')$ 
denotes the tangent map image of $\xi\in T_X\mathrm{Teich}(S)$.

\section{Comparison and convergence of virtual metrics}\label{Sec-two_lemmas}
In this section, we prove two lemmas regarding relation of virtual metrics
on the Teichm{\"u}ller space.
The first is an inequality due to Habermann and Jost \cite{Habermann--Jost}
in the case of ordinary metrics, and we only carry it onto the virtual version
(Lemma \ref{RS_leq_HJ}).
The second is an application of a well-known theorem due to D.~A.~Kazhdan \cite{Kazhdan_limit},
regarding convergence of the Bergmann metrics on finite coverings of Riemann surfaces
(Lemma \ref{HJ_lim_WP}).

\begin{lemma}\label{RS_leq_HJ}
	Let $S$ be a closed orientable surface of genus at least $2$.
	With respect to any finite regular cover $S'$ over $S$,
	and for all $X\in\mathrm{Teich}(S)$ and $\xi\in T_X\mathrm{Teich}(S)$,
	the following inequality holds,
	$$\|\xi\|_{\mathtt{RS}'}\leq
	2\times \|\xi\|_{\mathtt{HJ}'}.$$
\end{lemma}

\begin{proof}
		In the ordinary case, 
		the inequality $\|\xi\|_{\mathtt{RS}}\leq 2\times \|\xi\|_{\mathtt{HJ}}$
		is the same as \cite[Lemma 6.3]{Habermann--Jost},
		except a constant factor due to normalization;
		(check the formula (6.5) therein, 
		which has to agree with Royden's normalization 
		of the Siegel metric \cite[{Eq.~(15) on p.~397}]{Royden_invariant_metric},
		rather than Siegel's \cite[{Eq.~(2) on p.~3}]{Siegel_book}).
		For the reader's convenience,
		below we derive the inequality quickly from Lemmas \ref{HJ_theta} and \ref{RS_theta},
		following the same idea as in \cite{Habermann--Jost}.				
		
		Note that it suffices to prove the ordinary case $S'=S$, 
		otherwise arguing with $S'$	and dividing both sides by $[S':S]$.
		Fix an orthonormal basis of
		holomorphic differentials	$\theta_1,\cdots,\theta_p$ in $\Omega^1(X)$.
		Let $\mu=q/g_{\mathtt{B}}$ be a Bergmann harmonic Beltrami differential 
		representative of $\xi$.
		We denote by $\eta_j(\mu)\in \Omega^1(X)$ the orthogonal summand of 
		the $(1,0)$--differential $\overline{\mu\theta_j}\in A^{1,0}(X)$,
		namely,
		$$\eta_j(\mu)=\sum_l \theta_l\int_X \mu\theta_j\theta_l,$$
		for any $j=1,\cdots,p$.
		Since $\theta_1,\cdots,\theta_j$ are orthonormal, we obtain by Lemma \ref{RS_theta},
		$$\sum_j \int_X \eta_j(\mu)\overline{\eta_j(\mu)}
		=\sum_j \sum_k \left(\int_X \mu\theta_j\theta_k\right)\overline{\left(\int_X \mu\theta_j\theta_k\right)}
		=\|\xi\|_{\mathtt{RS}}^2\times 1/4,$$
		and by Lemma \ref{HJ_theta},
		$$\sum_j \int_X \eta_j(\mu)\overline{\eta_j}
		\leq \sum_j \int_X \mu\theta_j\overline{\mu\theta_j} =\|\xi\|_{\mathtt{HJ}}^2.$$
		Then the asserted inequality follows.		
\end{proof}

\begin{lemma}\label{HJ_lim_WP}
	Let $S$ be a closed orientable surface of genus at least $2$.
	Suppose $(S'_n\to S)_{n\in\Natural}$ is a sequence of finite regular covers
	converging to the universal cover.
	Then, for all $X\in\mathrm{Teich}(S)$ and $\xi\in T_X\mathrm{Teich}(S)$,
	the following convergence holds,
	$$\lim_{n\to\infty} \|\xi\|_{\mathtt{HJ}'_n}
	=\frac{1}{\sqrt{4\pi}}\times\|\xi\|_{\mathtt{WP}},$$
	where the notation $n$ in the subscript indicates the respective cover $S'_n$.
	Morover, the convergence is uniform on any compact set 
	of tangent vectors $(X,\xi)$ on $\mathrm{Teich}(S)$.
\end{lemma}

\begin{proof}
	For any compact Riemann surface $X$, 
	Kazhdan \cite{Kazhdan_limit} shows that 
	the sequence of virtual Bergmann metrics $g_{\mathtt{B}'_n}$
	on $X$ converges uniformly everywhere
	to the hyperbolic metric $g_{\mathtt{h}}$	rescaled by $1/4\pi$ 
	(that is, by $1/\sqrt{4\pi}$ in length);
	see \cite[Appendix]{McMullen_entropy} 
	for a clear statement and a short proof.
	See also Remark \ref{Kazhdan_remark} regarding the assumption
	on the covering sequence.
	
	For any cotangent vector $q$ in $T^*_X\mathrm{Teich}(S)=Q(X)$,
	the dual norm squares of $q$ converge as
	$$\lim_{n\to\infty} \|q\|_{\mathtt{HJ}'_n}^2
	=\lim_{n\to\infty} \int_X\frac{q\bar{q}}{g_{\mathtt{B}'_n}}
	=4\pi \times \int_X \frac{q\bar{q}}{g_{\mathtt{h}}}
	=4\pi \times\|q\|_{\mathtt{WP}}^2.$$
	This implies the asserted convergence of norms
	$$\lim_{n\to\infty} \|\xi\|_{\mathtt{HJ}'_n}=\frac{1}{\sqrt{4\pi}}\times \|\xi\|_{\mathtt{WP}},$$
	for any tangent vector $\xi\in T_X\mathrm{Teich}(S)$.
	
	Moreover, the convergence is uniform on any compact subsets of tangent vectors,
	because the norm $\|\_\|_{\mathtt{HJ}'_n}$ and $\|\_\|_{\mathtt{WP}}$
	are continuous functions on the Teichm\"uller tangent bundle.
\end{proof}

\begin{remark}\label{Kazhdan_remark}
	In the literature, Kazhdan's theorem is classically stated
	for cofinal \emph{towers} of regular finite covers of a compact Riemann surface $X$.
	However, we can replace `tower' with `sequence'.
	In fact, 
	as is evident from the proof in \cite[Appendix]{McMullen_entropy},
	the argument only relies on the properties 
	that $\pi_1(X)$ acts on all the covers $X'_n$,
	and that $X'_n$ are all compact,
	and that the hyperbolic injectivity radii of $X'_n$ tend to infinity.
\end{remark}

%
%
%

\section{Proof of the main theorem}\label{Sec-main_proof}
This section is devoted to the proof of Theorem \ref{main_vhev_WP}.

Let $f\colon S\to S$ be a pseudo-Anosov automorphism
on a connected closed orientable surface of genus at least $2$.	
Suppose that $(S'_n,f'_n)_{n\in\Natural}$ 
is a cofinal sequence of regular finite covers $(S'_n)_{n\in\Natural}$ of $S$
together with pseudo-Anosov automorphisms $f'_n\colon S'_n\to S'_n$ lifting $f$. 

Let $(S',f')=(S'_n,f'_n)$ be any term of the above sequence.
Denote by $J'$ the canonical composite map 
$\mathrm{Teich}(S)\to\mathrm{Teich}(S')\to\mathfrak{H}(S')$,
where the first map is the embedding induced by $S'\to S$.
By definition, the action of $f'\in\mathrm{Mod}(S')$ on $\mathrm{Teich}(S')$ 
preserves $\mathrm{Teich}(S)$, and 
extends the action of $f\in\mathrm{Mod}(S)$ on $\mathrm{Teich}(S)$.

For any regular path $\gamma\colon [0,1]\to \mathrm{Teich}(S)$ 
with the property $\gamma(1)=f.\gamma(0)$,
it follows that 
the induced action $f'_*$ on $\mathfrak{H}(S)$
also moves $J'(\gamma(0))$ to $J'(\gamma(1))$.
We obtain by Theorem \ref{tlen_Siegel},
$$\mathrm{Length}_{\mathtt{S}}(J'\circ\gamma)\geq
\tlen_{\mathtt{S}}(f'_*)=
2\times\sqrt{\mathbf{w}(P)},$$
where $P$ denotes the characteristic polynomial of $f'_*$ on $H_1(S';\Complex)$.
On the other hand, we obtain by Lemma \ref{RS_leq_HJ} and (\ref{RS_def}),
$$\mathrm{Length}_{\mathtt{S}}(J'\circ\gamma)
=\mathrm{Length}_{\mathtt{RS}'}(\gamma)\times \sqrt{[S':S]}
\leq 2\times \mathrm{Length}_{\mathtt{HJ}'}(\gamma)\times \sqrt{[S':S]}.$$

Retaining the subscript $n$, the above inequalities yield
$$\sqrt{\frac{\mathbf{w}(P_n)}{[S'_n:S]}}\leq \mathrm{Length}_{\mathtt{HJ}'_n}(\gamma).$$

Passing to limit, we obtain by Lemma \ref{HJ_lim_WP},
$$\varlimsup_{n\to\infty}\sqrt{\frac{\mathbf{w}(P_n)}{[S'_n:S]}}\leq\frac{\mathrm{Length}_{\mathtt{WP}}(\gamma)}{\sqrt{4\pi}}.$$

Apply the above inequality to a regular path $\gamma$ along the Weil--Petersson axis of $f$.
The Weil--Petersson translation length is realized as
$\tlen_{\mathtt{WP}}(f)=d_{\mathtt{WP}}(\gamma(0),\gamma(1))=\mathrm{Length}_{\mathtt{WP}}(\gamma)$.
In particular, we obtain the desired inequality,
$$\varlimsup_{n\to\infty}\sqrt{\frac{\mathbf{w}(P_n)}{[S'_n:S]}}\leq\frac{\tlen_{\mathtt{WP}}(f)}{\sqrt{4\pi}}.$$

This completes the proof of Theorem \ref{main_vhev_WP}.

\bibliographystyle{amsalpha}


\end{document}